\documentclass{amsart}
\usepackage[T1]{fontenc}
\usepackage[utf8]{inputenc}
\usepackage{amsmath}
\usepackage{array}
\usepackage{amsthm}
\usepackage{amssymb}
\input xy
\xyoption{all}

\usepackage[cmtip,all]{xy}

\usepackage{stmaryrd}

\usepackage{tikz}
\usetikzlibrary{matrix,arrows,decorations.pathmorphing}
\usepackage{tikz-cd}

\usepackage{enumerate}
\usepackage{mathtools}

\usepackage[french, english]{babel}

\usepackage{xcolor}

\newtheorem{thmintro}{Theorem}
\newtheorem{corintro}[thmintro]{Corollary}

\newtheorem{thm}{Theorem}[section]
\newtheorem{pr}[thm]{Proposition}

\newtheorem{lm}[thm]{Lemma}

\theoremstyle{definition}
\newtheorem{defi}[thm]{Definition}
\newtheorem{defi-prop}[thm]{Proposition-Definition}

\newtheorem{exintro}[thmintro]{Example}

\theoremstyle{remark}
\newtheorem{rk}[thm]{Remark}

\newcommand{\A}{{\mathcal{A}}}
\newcommand{\B}{{\mathcal{B}}}

\newcommand{\F}{{\mathcal{F}}}
\newcommand{\Pp}{{\mathcal{P}}}
\newcommand{\add}{\mathbf{Add}}

\newcommand{\I}{{\mathcal{I}}}
\newcommand{\K}{{\mathcal{K}}}
\newcommand{\LL}{{\mathcal{L}}}

\newcommand{\V}{{\mathcal{V}}}

\newcommand{\T}{{\mathcal{T}}}

\newcommand{\Md}{\text{-}\mathbf{Mod}}

\newcommand{\Si}{\mathfrak{S}}
\newcommand{\id}{\mathrm{id}}
\newcommand{\FF}{\mathbb{F}}

\newcommand{\res}{{\mathrm{res}}}

\newcommand{\op}{{\mathrm{op}}}
\newcommand{\Fp}{{\mathbb{F}_p}}

\newcommand{\Lan}{\mathrm{Lan}}
\newcommand{\colim}{\mathop{\mathrm{colim}}}
\newcommand{\Proj}{\mathbf{P}}

\newcommand{\GL}{\operatorname{GL}}

\newcommand{\Ext}{\mathrm{Ext}}
\newcommand{\Tor}{\mathrm{Tor}}
\newcommand{\Hom}{\mathrm{Hom}}

\newcommand{\kk}{\Bbbk}
\newcommand{\PP}{\mathcal{P}}

\begin{document}

\title{The homology of additive functors in prime characteristic}

\author[A. Djament]{Aur\'elien Djament}
\address{Institut Galil\'ee, Universit\'e Sorbonne Paris Nord, 99 avenue Jean-Baptiste Cl\'ement, 93430 Villetaneuse, France.}
\email{djament@math.univ-paris13.fr}
\urladdr{https://djament.perso.math.cnrs.fr/}
\thanks{A. Djament is partly supported by the projects ChroK (ANR-16-CE40-0003), AlMaRe (ANR-19-CE40-0001-01) and Labex CEMPI (ANR-11-LABX-0007-01)}

\author[A. Touz\'e]{Antoine Touz\'e}
\address{Univ. Lille, CNRS, UMR 8524 - Laboratoire Paul Painlev\'e, F-59000 Lille, France}
\email{antoine.touze@univ-lille.fr}
\urladdr{https://pro.univ-lille.fr/antoine-touze/}
\thanks{A. Touz\'e is partly supported by the project ChroK (ANR-16-CE40-0003) and Labex CEMPI (ANR-11-LABX-0007-01)}

\begin{abstract}
We express certain $\Ext$ and $\Tor$ groups in the category of all functors from an $\Fp$-linear additive category $\A$ to vector spaces in terms of $\Ext$ and $\Tor$ computed in the full subcategory of additive functors from $\A$ to vector spaces.
It may be applied to group homology computations for general line groups. 
\end{abstract}

\maketitle

\selectlanguage{french}
\renewcommand{\abstractname}{R\'esum\'e}
\begin{abstract}
Nous exprimons certains groupes d'extensions et de torsion dans la cat\'egorie de tous les foncteurs depuis une cat\'egorie additive $\Fp$-lin\'eaire vers les espaces vectoriels en termes d'$\Ext$ et de $\Tor$ calculés dans la sous-catégorie pleine des foncteurs additifs de $\A$ vers les espaces vectoriels. Cela peut s'appliquer à des calculs d'homologie des groupes linéaires.
\end{abstract}

\selectlanguage{english}

\section{Introduction}

Let $\rho$ and $\pi$ be two additive functors from an essentially small additive category $\A$ to the category $\V_\kk$ of $\kk$-vector spaces. There are two natural ways to define the higher extensions between $\rho$ and $\pi$. Firstly, we can consider $\pi$ and $\rho$ as objects of the abelian category $\add(\A,\kk)$ of additive functors from $\A$ to $\kk$-vector spaces. Secondly, we can consider $\pi$ and $\rho$ as objects of the abelian category $\F(\A,\kk)$ of \emph{all functors} from $\A$ to $\kk$-vector spaces.
Since $\add(\A,\kk)$ is a full abelian subcategory of $\F(\A,\kk)$, there is a canonical map of graded vector spaces:
\begin{align*}
\Phi:\Ext^*_{\add(\A,\kk)}(\pi,\rho)\to \Ext^*_{\F(\A,\kk)}(\pi,\rho)\;.
\end{align*}
Note that when
$\A=\Proj_R$, the category of finitely generated projective right modules over a ring $R$, every additive functor is isomorphic to $t_M:P\mapsto P\otimes_R M$ for some $(R,\kk)$-bimodule $M$, and there is a graded isomorphism:  
\[\Ext^*_{\add(\Proj_R,\kk)}(t_M,t_N)\simeq \Ext^*_{R\otimes_{\mathbb{Z}}\kk}(M,N)\;.\] 
Thus the source of $\Phi$ be computed by standard homological algebra in module categories. In sharp contrast, $\F(\Proj_R,\kk)$ is not equivalent to a module category if $R\ne 0$, which makes the target of $\Phi$ a bit more mysterious; in some cases it is related to topological Hochschild homology \cite{PiraWald}.

If $\kk$ is a field of characteristic zero, then $\Phi$ is an isomorphism by \cite[Theorem 1.2]{Dja-Ext-Pol}. This is no longer the case if $\kk$ is a field of positive characteristic. For example, if $\kk$ is a finite field and if $\A=\Proj_\kk$ then the ring $\kk\otimes_\mathbb{Z}\kk$ is semi-simple. Hence the source of $\Phi$ is always zero in positive degrees, whereas the fundamental computation of \cite[Theorem 0.1]{FLS} shows that the target of $\Phi$  may be nonzero in infinitely many positive degrees. 

Let us introduce a few notations to formulate our first theorem. Let $I$ denote the embedding $\Proj_\kk\hookrightarrow \V_\kk$. We consider $\Ext^*_{\F(\Proj_\kk,\kk)}(I,I)$ as a graded algebra for Yoneda composition.

Notice furthermore that every functor $\pi:\A\to \V_\kk$ yields an  exact restriction functor $\pi^*:\F(\Proj_\kk,\kk)\to\F(\A,\kk)$, which sends $F$ to the composition $\pi^*F = \overline{F}\circ \pi$, where $\overline{F}:\V_\kk\to \V_\kk$ denotes the left Kan extension of $F$ to the category of all vector spaces, that is, $\overline{F}(v)=\mathrm{colim}\, F(u)$, with colimit taken over the filtered poset of finite dimensional subspaces $u$ of $v$, ordered by inclusion.
We therefore have a graded map:
\[\pi^*:\Ext^*_{\F(\Proj_\kk,\kk)}(F,G)\to  \Ext^*_{\F(\A,\kk)}(\pi^*F,\pi^*G)\;.\]

We can now define, for all additive functors $\pi, \rho: \A\to\V_\kk$, a linear map (here `$\circ$' stands for Yoneda composition)
\begin{align*}\tag{$\dagger$}\label{ppsi}
\Psi:\Ext^*_{\add(\A,\kk)}(\pi,\rho)\otimes_\kk \Ext^*_{\F(\Proj_\kk,\kk)}(I,I) & \to\Ext^*_{\F(\A,\kk)}(\pi,\rho)\\
e\otimes e' & \mapsto\Phi(e)\circ \pi^*e'\;.
\end{align*}

\begin{thmintro}\label{thm-intro-3}
Let $\kk$ be a perfect field of positive characteristic $p$. There is an isomorphism of graded $\kk$-algebras
\[\Ext^*_{\F(\Proj_\kk,\kk)}(I,I) \simeq \kk[e_1,e_2,e_3,\dots]/\langle e_1^p, e_2^p,e_3^p \dots\rangle\;,\]
where each class $e_i$ has cohomological degree $2p^{i-1}$.
If in addition the essentially small additive category $\A$ is $\Fp$-linear, then the map
\begin{align*}
\Psi:\Ext^*_{\add(\A,\kk)}(\pi,\rho)\otimes_\kk \Ext^*_{\F(\Proj_\kk,\kk)}(I,I)\xrightarrow{\simeq} \Ext^*_{\F(\A,\kk)}(\pi,\rho)
\end{align*}
defined by $(\dagger)$ is a graded $\kk$-linear isomorphism for all additive functors $\pi$ and $\rho$.
\end{thmintro}
If $\kk$ is a finite field, the computation of $\Ext^*_{\F(\Proj_\kk,\kk)}(I,I)$ given in this theorem is nothing but \cite[Theorem 0.1]{FLS}.  
Although technically  different, theorem \ref{thm-intro-3} is similar in spirit to the main theorems of \cite{LL,LLbis} (see also \cite[Corollary 4.2]{PiraSpectral}) which compare Hochschild homology and topological Hochschild homology of smooth $\Fp$-algebras. 

The calculations of \cite{PiraZnZ} show that $\Psi$ may fail to be an isomorphism if $\A$ is not $\Fp$-linear, for example when $\A=\mathbf{P}_{\mathbb{Z}/p^r}$ for $r>1$ and $\pi=\rho$ is reduction modulo $p$. Remark \ref{rk-perfect} shows that the isomorphism may fail if $\kk$ is not perfect.

The proof of theorem~\ref{thm-intro-3} consists of successive reductions from the classical case of $\A=\Proj_\Fp$ \cite{FLS}. It relies on several well-known tools as adjunctions between exact functors, which are propagated to higher $\Ext$, comparison of $\delta$-functors, or extension of scalars at the target (from functors to $\Fp$-vector spaces to functors to $\kk$-vector spaces). We use \textit{strict} polynomial functors \cite{FS} as intermediate functor categories for which base change behaves very nicely in cohomology,  to prove lemma~\ref{lm-step-1}, the first step of the proof of the theorem. We use the perfectness assumption on $\kk$ through the Hochschild-Kostant-Rosenberg theorem (which implies that $HH^*(\kk)$ is trivial) to prove lemma~\ref{lm-iso}, using the identification between additive functors in $\F(\Proj_\kk,\kk)$ and $(\kk,\kk)$-bimodules. The last ingredient to prove theorem~\ref{thm-intro-3} is the use of the $\aleph$-additive envelope $\A^\aleph$ of $\A$ (see section~\ref{sect-env}), for a big enough regular cardinal $\aleph$, which is roughly speaking obtained from the additive category $\A$ by formally adding $\aleph$-small coproducts. We need it for an important adjunction argument because the representable functors $\A(a,-)$ may take values in \textit{infinite-dimensional} $\Fp$-vector spaces, so we pass through $\Proj_\Fp^\aleph$ and $\A^\aleph$ (see the proof of lemma~\ref{lm-step-2}).

\subsection*{Comparison of $\Tor$}
If $\kk$ is an arbitrary field and $\A$ an arbitrary essentially small additive category, there is a tensor product over $\A$ \cite[IX.6]{ML}:
\[
\underset{\kk[\A]}{\otimes} :\; \F(\A^\op,\kk)\times \F(\A,\kk) \to   \V_\kk\;,
\]
with left derived functors denoted by $\Tor_*^{\kk[\A]}(F,G)$. 
One can also consider the restriction of the tensor product to the full subcategories of additive functors, whose left derived functors yield an "additive $\Tor$" denoted by $\Tor^{\kk\otimes_\mathbb{Z}\A}(F,G)$ because an additive functor $\A\to\V_\kk$ is the same as a $\kk$-linear functor $\kk\otimes_\mathbb{Z}\A\to\V_\kk$.

Note that if $\A=\Proj_R$, the additive $\Tor$ can be computed in terms of modules. Indeed, every contravariant additive functor from $\A$ to vector spaces has the form $h_M:P\mapsto \Hom_R(P,M)$ for some $(\kk,R)$-bimodule $M$ and there is a graded isomorphism: 
\begin{align}
\Tor^{\Proj_\kk\otimes_{\mathbb{Z}}\kk}(h_M,t_N)\simeq \Tor^{R\otimes_{\mathbb{Z}}\kk}_*(M,N)\;.\label{eq-add-Tor}
\end{align}

If $\kk$ has characteristic zero, then $\Tor^{\kk\otimes_\mathbb{Z}\A}(\pi,\rho)$ is isomorphic to $\Tor_*^{\kk[\A]}(\pi,\rho)$ for all additive functors $\pi$ and $\rho$, but this is not the case in positive characteristic. In fact, we can dualize theorem \ref{thm-intro-3} to obtain the following result. 
\begin{corintro}\label{cor-intro}
Let $\kk$ be a perfect field of positive characteristic and let $\A$ be an essentially small additive $\Fp$-linear category. Let $T_*$ denote the graded vector space equal to $\kk$ in every even degree, and to zero in every odd degree. There is a graded isomorphism, natural with respect to the additive functors $\pi$ and $\rho$:
\[ \Tor_*^{\kk[\A]}(\pi,\rho)\simeq\Tor_*^{\kk\otimes_{\mathbb{Z}}\A}(\pi,\rho)\otimes_\kk T_*\;.\]
\end{corintro}

\subsection*{$\Tor$ between non additive functors.} Corollary \ref{cor-intro} should be thought as a computation of $\Tor$ in the category of all functors, based on the simpler computation of $\Tor$ in the category of additive functors. 

In section \ref{sec-comput-tor}, we explain how to combine corollary \ref{cor-intro} with standard techniques of computations  in functor categories to compute $\Tor$ between certain non-additive functors. We give an elementary illustration of this in the next example. Let $V$ be a finite-dimensional $\kk$-linear representation of $\Si_d$, and let $F_V$ be the endofunctor of vector spaces defined by 
$F_V(E)=E^{\otimes d}\otimes_{\Si_d} V$, where $\Si_d$ acts on $E^{\otimes d}$ by permuting the factors of the tensor product. We extend $F_V$ to an endofunctor $\widetilde{F}_V$ of graded vector spaces by letting $\Si_d$ act on the tensor product of graded vector spaces with a Koszul sign. 

\begin{exintro}\label{ex-intro}
Let $\kk$ be a perfect field of positive characteristic $p$ and let $\A$ be an $\Fp$-linear additive category. Assume that $d$ is invertible in $\kk$. For all additive functors $\pi$ and $\rho$, there is a graded isomorphism of vector spaces
\[\Tor_*^{\kk[\A]}(\pi^*S^d,\rho^*F_V)\simeq \widetilde{F}_V( \,\Tor_*^{\kk\otimes_{\mathbb{Z}}\A}(\pi,\rho)\otimes_\kk T_*\,)\;,\]
where $S^d$ denotes the $d$th symmetric power functor.
\end{exintro} 
If $d=1$, this example gives back corollary \ref{cor-intro}, whereas the functors $\pi^*S^d$ and $\rho^*F_V$ are not additive if $d\ge 2$: they are polynomial of degree $d$ in the sense of Eilenberg and Mac Lane \cite{EML}. Additional and more general $\Tor$-computations are given in section \ref{sec-comput-tor}. However, our computations typically require that the characteristic of $\kk$ is big enough. Computations in small characteristic are much more difficult, we treat them in \cite{DT-HIII}.

\subsection*{Application to the homology of general linear groups.}
We finish this introduction by recalling one of the main motivations of the present work, namely the connection with the homology of general linear groups. 

Every functor $F:\Proj_R^\op\to \V_\kk$ determines a sequence of nested right $\kk$-linear representations $F(R^n)$ of $\GL_n(R)$, in which an element $g$ of the general linear group acts as $F(g)$ on the vector space $F(R^n)$. These actions assemble into a right $\kk$-linear action of $\GL_\infty(R)=\bigcup_{n\ge 1} \GL_n(R)$ on  $F_\infty = \bigcup_{n\ge 1} F(R^n)$. Similarly, every functor $G:\Proj_R\to \V_\kk$ determines a left $\kk$-linear representation $G_\infty$ of $\GL_\infty(R)$.

It follows from \cite{DjaR} that when the functors $F$ and $G$ are polynomial in the sense of Eilenberg and Mac Lane, there is a graded isomorphism :
 \begin{align}H_*(\GL_{\infty}(R),F_\infty\otimes_\kk G_\infty)\simeq H_*(\GL_\infty(R),\kk)\otimes_\kk \Tor_*^{\kk[\Proj_R]}(F,G)\;\;. \label{eq-iso-Dja}
\end{align}

Typical examples of polynomial functors are the functors $\pi^*S^d$ and $\rho^*F_V$ of example \ref{ex-intro}, with $\pi=h_M$ and $\rho=t_N$. The corresponding representations $(h_M^*S^d)_\infty$ and $(t_N^*F_V)_\infty$ of $\GL_\infty(R)$ are concretely described by
\[(h_M^*S^d)_\infty = S^d(M^\infty)\;,\quad (t_N^*F_V)_\infty = F_V(N^\infty) = (N^\infty)^{\otimes d}\otimes_{\Si_d}V\;,\]
where $N^\infty$ and $M^\infty$ are the standard representations of $\GL_\infty(R)$ associated to the $(R,\kk)$-bimodule $N$ and the $(\kk,R)$ bimodule $M$. In other words, $N^\infty =\bigoplus_{k>0} N_k$ is a countable direct sum of copies of $N$, and a matrix $[r_{ij}]$ of $\GL_\infty(R)$ acts as the endomorphism of $N^\infty$ whose component $N_j\to N_i$ is left multiplication by $r_{ij}$; similarly, $M^\infty =\bigoplus_{k>0} M_k$ is a countable direct sum of copies of $M$, and a matrix $[r_{ij}]$ of $\GL_\infty(R)$ acts as the endomorphism of $M^\infty$ whose component $M_j\to M_i$ is right multiplication by $r_{ji}$.
Combining example \ref{ex-intro} and isomorphism \eqref{eq-add-Tor} gives a fairly concrete computation of the $\Tor$-term  appearing in the right hand-side of \eqref{eq-iso-Dja}.
\begin{exintro}
Let $\kk$ be a perfect field of positive characteristic $p$, let $R$ be a ring of characteristic $p$, let $V$ be a $\Si_d$-module with $d<p$ and let $M$ and $N$ be a right, respectively left, $R\otimes_{\mathbb{Z}}\kk$-module. 
There is a graded isomorphism, where $T_*$ is a graded vector space which has dimension one in even degrees and dimension $0$ in odd degrees:
\[H_*(\GL_{\infty}(R),S^d(M^\infty)\otimes_\kk F_V(N^\infty))\simeq H_*(\GL_\infty(R),\kk)\otimes_\kk \widetilde{F}_V(\Tor_*^{R\otimes_{\mathbb{Z}}\kk}(M,N)\otimes_\kk T_*).\]
\end{exintro}
Further examples of computations can be obtained by combining isomorphism \eqref{eq-add-Tor}, corollary \ref{cor-intro} and the computations of section \ref{sec-comput-tor}.

\section{Notation}

Throughout the article, $\kk$ is a field, and $\A$ is an additive category \cite[VIII.2]{ML} which is essentially small.

We denote by $\V_\kk$ the category of $\kk$-vector spaces and $\kk$-linear maps, and if $R$ is a ring, we denote by $\Proj_R$ the category of finitely projective right $R$-modules and $R$-linear maps. Thus $\Proj_\kk$ is the category of finite-dimensional $\kk$-vector spaces.

We let $\F(\A,\kk)$ be the category of all functors from $\A$ to $\V_\kk$ and natural transformations. (Note that there is a \emph{set} of natural transformations between two functors because $\A$ is essentially small.) This is a bicomplete abelian category with enough injectives and projectives. We let $\add(\A,\kk)$ be the full subcategory of $\F(\A,\kk)$ on the additive functors. This category is stable under limits and colimits, and also has enough projectives and injectives. We refer the reader to \cite{Mi72} and \cite{Pira-Pan} for an introduction to homological algebra in these functor categories.

We let $\PP_\kk$ be the category of strict polynomial functors of bounded degree over $\kk$ as in \cite{FS}. Although this category does not appear in the statement of theorem~\ref{thm-intro-3} it plays a role in its proof. Strict polynomial functors over $\kk$ are highly structured functors $F:\Proj_\kk\to \V_\kk$, and there is an exact forgetful functor $\PP_\kk\to \F(\Proj_\kk,\kk)$ which forgets the "strict polynomial structure". If $\kk$ is an infinite field, this forgetful functor is fully faithful.

\section{$\aleph$-additive envelopes}\label{sect-env}
Our proof of theorem \ref{thm-intro-3} will rely on a variation of the classical notion of the additive envelope of a category. 
 
Let $\aleph$ be a regular cardinal and let $k$ be a commutative ring. A $k$-linear category is \emph{$\aleph$-additive} if it has all \emph{$\aleph$-direct sums}, that is, if all direct sums indexed by sets of cardinality less than $\aleph$ exist. A $k$-linear functor between two such categories is \emph{$\aleph$-additive} if it preserves $\aleph$-direct sums. 
\begin{defi}\label{def-kappa-completion}
Given a $k$-linear category $\K$ with null object, we let $\K^\aleph$ be the category 
whose objects are the families of objects of $\K$ indexed by sets of cardinality less or equal to $\aleph$. Such an object is denoted by a formal direct sum $\bigoplus_{i\in\I}x_i$. The morphisms $f:\bigoplus_{j\in\mathcal{J}}x_j\to \bigoplus_{i\in\mathcal{I}}y_i$ are the `matrices' $[f_{ij}]_{(i,j)\in \I\times\mathcal{J}}$ with entries $f_{ij}\in \K(x_j,y_i)$ and such that for all $j_0$ only a finite number of morphisms $f_{ij_0}$ are nonzero. The composition of morphisms is given by matrix multiplication.

We identify $\K$ with the full subcategory of $\K^{\aleph}$ on the formal direct sums with only one object, and we let $\iota:K\hookrightarrow \K^\aleph$ be the inclusion. The category $\K^{\aleph}$ is called the \emph{$\aleph$-additive envelope} of $\K$.
\end{defi}
The definition of morphisms in $\K^\aleph$ shows that the formal direct sum $\bigoplus_{i\in\I}x_i$ is the categorical coproduct of the $x_i$ in $\K^\aleph$, which justifies the direct sum notation; it is also the categorical product if $\I$ is finite. Point \eqref{item-n2} of the next elementary proposition justifies the name given to $\K^\aleph$.
\begin{pr}\label{pr-elt-prop}
The category $\K^\aleph$ is $k$-linear and $\aleph$-additive. Moreover:
\begin{enumerate}
\item\label{item-n1} An object $x$ of $\K^\aleph$ is isomorphic to a finite direct sum of objects of $\K$ if and only if $\K^\aleph(x,-):\K^\aleph\to k\Md$ is $\aleph$-additive, 
\item\label{item-n2} Every $k$-linear functor $F:\K\to \LL$ with $\aleph$-additive codomain extends to a unique (up to isomorphism) $\aleph$-additive functor $F^\aleph:\K^\aleph\to \LL$.
\item\label{item-n3} If $\K$ is essentially small, then $\K^\aleph$ is essentially small.
\end{enumerate}
\end{pr}
\begin{proof}
\eqref{item-n1} The objects of $\K$ (hence their finite direct sums) are $\aleph$-additive by the definition of morphisms in $\K^\aleph$. Conversely, if $x=\bigoplus x_i$ is an object such that $\K^\aleph(x,-)$ is $\aleph$-additive,  the isomorphism $\K^\aleph(x,\bigoplus x_i)\simeq\bigoplus \K^\aleph(x,x_i)$ shows that $\id_x$ factors through a finite direct sum of the $x_i$, hence $x$ is isomorphic to a finite direct sum of objects of $\K$. 

\eqref{item-n2} For all objects $x=\bigoplus x_i$ we choose a direct sum $\bigoplus F(x_i)$ in $\LL$. The assignment $F^\aleph(x)=\bigoplus F(x_i)$ defines an $\aleph$-additive functor such that $F^\aleph\circ \iota = F$. Uniqueness follows from the fact that given any pair of $\aleph$-additive functors $F',G':\K^\aleph\to \LL$, every natural transformation $\theta$ between their restrictions to $\K$ extends uniquely into a natural transformation $\theta':F'\to G'$. To be more specific, for all formal direct sums $\bigoplus x_i$, the map $\theta'_{\bigoplus x_i}$ is the composition
\[ \textstyle F'(\bigoplus x_i) \xleftarrow[]{\simeq} \bigoplus F(x_i) \xrightarrow[]{\bigoplus \theta_{x_i}}\bigoplus G(x_i) \xrightarrow[]{\simeq}G'(\bigoplus x_i)\;.\]
\eqref{item-n3} The set of isomorphism classes of $\K^\aleph$ has cardinal less than or equal to that of $\K$ and $\aleph$.
\end{proof}

For example, using the universal property \eqref{item-n2}, $\Proj_\kk^\aleph$ is easily seen to be equivalent to the category of $\kk$-vector spaces of dimension less than $\aleph$.

We now come to the usefulness of additive envelopes when computing $\Ext$ in functor categories. We are interested in the case where $k=\kk$ is a field and $\K=\A$. 

Let $F:\A\to \V_\kk$ be a functor, possibly non-additive. 
If $x=\bigoplus_{i\in I} x_i$ is an object of $\A^\aleph$, then the left Kan extension of $F$ along $\iota:\A\to \A^\aleph$ can be computed by 
\[\Lan_\iota F(x)=\colim_{y\in \T}F(y)\;,\]
where $\T$ denotes the poset of finite direct sums $\bigoplus_{i\in J}x_i$, ordered by canonical inclusions. (Indeed $\T$ is cofinal in the comma category $\K\downarrow x$.) This implies that if $F$ is additive then $\Lan_\iota F$ is isomorphic to $F^\aleph$, hence additive. Since $\T$ is filtered, this also implies that $\Lan_\iota:\F(\A,\kk)\to \F(\A^\aleph,\kk)$ is exact. The following result follows by standard homological algebra (see e.g. \cite[Lemma 1.4(v)]{Pira-Pan}).
\begin{pr}\label{pr-enlarge}
For all functors $F:\A\to \V_\kk$ and $G:\A^\aleph\to \V_\kk$, restriction along $\iota:\A\hookrightarrow \A^\aleph$ yields a graded isomorphism:
\[\Ext^*_{\F(\A^\aleph,\kk)}(\Lan_\iota F,G)\simeq  \Ext^*_{\F(\A,\kk)}(F,G\circ\iota)\;.\]
Moreover, if $F$ and $G$ are additive then a similar isomorphism also holds for $\Ext$ computed in the respective categories of additive functors.
\end{pr}

The next proposition reveals the advantage of working with the larger category $\A^\aleph$.
\begin{pr}\label{pr-transport}
Assume that for all pairs $(a,b)$ of objects of $\A$, $\A(a,b)$ is an $\Fp$-vector space of cardinal less than $\aleph$. Then for all $a$ in $\A$, the functor $\A^\aleph(a,-):\A^\aleph\to \Proj_\Fp^\aleph$ has a left adjoint $a\otimes -$. Moreover, for all functors $F:\Proj_\Fp\to \V_\kk$ and $G:\A^\aleph\to \V_\kk$, there is a graded isomorphism:
\[\Ext^*_{\F(\A^\aleph,\kk)}(F\circ \A^\aleph(a,-),G)\simeq \Ext^*_{\F(\Proj_\Fp,\kk)}(F,G\circ (a\otimes-))\;. \]  
which is induced by precomposition by $a\otimes -$ and restriction along the unit of adjunction $x\mapsto \A(a,a\otimes x)$. If $F$ and $G$ are additive then a similar isomorphism also holds for $\Ext$ computed in the categories of additive functors.
\end{pr}
\begin{proof}

Let $v$ be a $\Fp$-vector space with basis $(b_i)_{i\in I}$. We let $v\otimes x:=\bigoplus_{i\in I}x_i$ where each $x_i$ denotes a copy of $x$. The $\Fp$-linear map $v\mapsto \A^\aleph(x,v\otimes x)$, sending $b_i$ to the canonical inclusion of $x$ as the factor $x_i$ of $v\otimes x$ is initial in $v\downarrow \A^\aleph(x,-)$, and the existence of the adjoint follows from \cite[Lemma 4.6.1]{Riehl}. The $\Ext$-isomorphism is then given by \cite[Lemma 1.5]{Pira-Pan}.
\end{proof}

\section{Recollections on Frobenius twists}\label{subsec-Frobenius}
Assume that the field $\kk$ is perfect, of positive characteristic $p$.
For all integers $r$, the $r$-th Frobenius twist of a vector space $v$ is the  vector space $v^{(r)}$ which equals $v$ as an abelian group, with action of $\kk$ given by 
$\lambda\cdot x:= \lambda^{p^{-r}}x$.
This construction, which may also be seen as \textit{tensoring} with $\kk$ with action twisted by $\lambda\mapsto\lambda^{p^r}$, is natural with respect to the vector space $v$. If $r\ge 0$ then the $r$-th Frobenius twist is the underlying functor of a strict polynomial functor $I^{(r)}$. 

We denote by $E_r^*$ the graded $\kk$-algebra of self-extensions of $I^{(r)}$ in strict polynomial functors. This algebra is computed in \cite[Theorem 4.10]{FS} (with a slightly different notation for the generators):
\[E_r^*=\Ext^*_{\mathcal{P}_\kk}(I^{(r)},I^{(r)})\simeq \kk[e_1,\dots,e_r]/\langle e_1^p =e_2^p=\dots = e_r^p=0\rangle\;,\]
where $e_i$ is a class of cohomological degree $2p^{i-1}$. In particular, as a graded vector space, $E^*_r$ has dimension one in degree $2i$ for $0\le i<p^r$, and it is zero in other degrees. 
Precomposition by the Frobenius twist $I^{(1)}$ yields a morphism of graded $\kk$-algebras:
$ E_r^*\to E_{r+1}^*.$
It is shown in \cite[Corollary 4.9]{FS} that this map is injective. For dimension reasons, it is an isomorphism in degrees less than $2p^r$. We denote by $E_\infty^*$ the colimit of the $E_r^*$. Thus, as graded $\kk$-algebras we have
\[E_\infty^*\simeq \kk[e_i, i\ge 1]/\langle e_i^p =0 \text{ for all $i\ge 1$}\rangle\]
and the canonical map $E_r^*\to E_\infty^*$ identifies $E_r^*$ with the subalgebra of $E_\infty^*$ generated by the classes $e_1,\dots,e_r$. 

For all $r\ge 0$, we define a morphism of graded $\kk$-algebras $\alpha_r^\kk$ as the composition of the $\Ext$-map induced by the forgetful functor $\PP_\kk\to \F(\Proj_\kk,\kk)$ with the $\Ext$-map induced by precomposition by the $(-r)$-th Frobenius twist (here we need that $\kk$ is perfect!):
\[\alpha_r^\kk: E_r^*=\Ext^*_{\mathcal{P}_\kk}(I^{(r)},I^{(r)})\to \Ext^*_{\F(\Proj_\kk,\kk)}(I^{(r)},I^{(r)})\to \Ext^*_{\F(\Proj_\kk,\kk)}(I,I)\;.\]
Passing to the colimit yields a morphism of graded $\kk$-algebras: 
\[\alpha_\infty^\kk:E_\infty^*\to \Ext^*_{\F(\Proj_\kk,\kk)}(I,I)\;.\] 
The following result is a special case of \cite[Theorem 3.10]{FFSS}. It can also be proved directly by comparing \cite{FLS} (which computes the target of $\alpha^{\FF_p}_r$) and \cite[§\,4]{FS} (which computes the source of $\alpha^{\FF_p}_r$ by following the methods of \cite{FLS}, adapted to $\Pp$, making easy to compare both sides). 
\begin{thm}[\cite{FLS,FS}]\label{thm-FFLSS}
If $\kk=\Fp$, the map $\alpha_\infty^\Fp$ is an isomorphism. Equivalently, for all integers $r\ge 0$, the maps $\alpha_r^\Fp$ are isomorphisms in degrees less than $2p^r$.
\end{thm}

\section{A first $\Ext$-computation in $\F(\A,\kk)$}
In this section, $\kk$ is a perfect field of positive characteristic $p$.
Given an additive functor $\pi:\A\to\V_\kk$ and $r\in \mathbb{N}\cup \{+\infty\}$, we denote by $\pi^*_r$ the morphism of graded $\kk$-algebras defined as the composition:
\[\pi^*_r: E^*_r\xrightarrow[]{\alpha_r^\kk} \Ext^*_{\F(\Proj_\kk,\kk)}(I,I)\xrightarrow[]{\pi^*} \Ext^*_{\F(\A,\kk)}(\pi,\pi)\] 
where $\alpha_r$ is as in section \ref{subsec-Frobenius} and $\pi^*$ is induced by evaluation on $\pi$. Recall that 
\[\Phi: \Ext_{\add(\A,\kk)}^*(\pi,\rho)\to \Ext_{\F(\A,\kk)}^*(\pi,\rho)\] is the canonical map induced by the exact inclusion $\add(\A,\kk)\to\F(\A,\kk)$. The goal of this section is to prove the following result. It is the main step towards theorem \ref{thm-intro-3}, for which we will then have only to identify $E^*_\infty$.
\begin{thm}\label{thm-prelim}
Let $\kk$ be a perfect field of positive characteristic $p$.
Assume that $\A$ is $\Fp$-linear. For all additive functors $\pi,\rho:\A\to \V_\kk$, the following map is an isomorphism of graded $\kk$-vector spaces:
\[\begin{array}{cccc}
\Psi_\infty:&\Ext_{\add(\A,\kk)}^*(\pi,\rho)\otimes_\kk E^*_\infty & \to & \Ext_{\F(\A,\kk)}^*(\pi,\rho)\\
 &x\otimes y & \mapsto & \Phi(x) \circ \pi^*_\infty(y)
\end{array}\;.
\]
\end{thm}

We shall prove theorem \ref{thm-prelim} in several steps.
\begin{lm}[First step]\label{lm-step-1}
Theorem \ref{thm-prelim} holds if $\A=\Proj_\Fp$.
\end{lm}
\begin{proof}
There is an equivalence of categories $\V_\kk\simeq \add(\Proj_\Fp,\kk)$ which sends a $\kk$-vector space $u$ to the additive functor $v\mapsto v\otimes_\Fp u$.
Thus the additive functors $\pi$ and $\rho$ are direct sums of copies of the additive functor $t:v\mapsto v\otimes_\Fp \kk$. Hence by additivity of $\Ext$, $\Phi$, $\pi^*_\infty$, and by degreewise finiteness of $E^*_\infty$, the proof reduces to the case $\pi=\rho=t$. In this case we have
\[\Ext^*_{\add(\A,\kk)}(t,t)=\Ext^*_{\kk}(\kk,\kk)=\Hom_\kk(\kk,\kk)\simeq \kk\;.\]
Thus the proof reduces to showing that $t^*_\infty:E^*_\infty\to \Ext^*_{\F(\Proj_\Fp,\kk)}(t,t)$ is an isomorphism. The latter is equivalent to showing that for all positive integers $r$, $t^*_r$ is an isomorphism in degree less than $2p^r$.

For this purpose, we consider the commutative square
\[\begin{tikzcd}
\Ext^i_{\PP_\kk}(I^{(r)},I^{(r)})\ar{rr}{t_r^*}&& \Ext^i_{\F(\Proj_\Fp,\kk)}(t,t)\\
\kk\otimes_\Fp\Ext^i_{\PP_\Fp}(I^{(r)},I^{(r)})\ar{u}{\simeq}\ar{rr}{\kk\otimes\alpha_r^\Fp}&& \kk\otimes_\Fp\Ext^i_{\F(\Proj_\Fp,\Fp)}(I,I)\ar{u}{\simeq} 
\end{tikzcd}\]
where the vertical isomorphism on the left is the base change isomorphism for strict polynomial functors \cite[2.7]{SFB} and the vertical map on the right is induced by tensorization with $\kk$, which is here an isomorphism thanks to \cite[Proposition~2.6 and Theorem~2.7]{Pira-Pan} (which reads again \cite[Proposition 10.1]{FLS}). As recalled in section \ref{subsec-Frobenius} $\alpha^\Fp_r$ is an isomorphism in degrees less than $2p^r$, hence $t_r^*$ is an isomorphism in degrees less than $2p^r$.
\end{proof}

\begin{lm}[Second step]\label{lm-step-2} Let $\A$ be an arbitrary essentially small, additive, $\Fp$-linear category. Then theorem \ref{thm-prelim} holds if $\pi=\kk\otimes_\Fp\A(a,-)$ for some $a\in \A$.
\end{lm}
\begin{proof}
Let $\aleph$ be a cardinal larger than the cardinal of $\A(x,y)$ for all $x$ and $y$, and let $\A^\aleph$ be the $\aleph$-additive enveloppe of $\A$, as in definition \ref{def-kappa-completion}. Let $\pi'=k\otimes_{\mathbb{Z}}\A^\aleph(a,-)$ and let $\rho':\A^\aleph\to k\Md$ be an arbitrary additive extension of $\rho$. Then $\pi'$ is $\aleph$-additive, hence it is the left Kan extension of $\pi$ to $\A^\aleph$. Thus, by proposition \ref{pr-enlarge}, lemma \ref{lm-step-2} reduces to proving that
\[\Psi_\infty: \Ext^*_{\add(\A^\aleph,\kk)}(\pi',\rho')\otimes_\kk E^*_\infty \to \Ext^*_{\F(\A^\aleph,\kk)}(\pi',\rho')\]
is an isomorphism. 
Now let $t':\Proj_{\Fp}^\aleph\to \V_\kk$ be given by $t'(v)=\kk\otimes_\Fp v$. Then proposition \ref{pr-transport} provides a commutative square with vertical isomorphisms:
\begin{equation*}
\begin{tikzcd}
\Ext^*_{\add(\A^\aleph,\kk)}(\pi',\rho')\otimes_\kk E^*_\infty \ar{r}{\Psi_\infty}\ar{d}{\simeq} & \Ext^*_{\F(\A^\aleph,\kk)}(\pi',\rho')\ar{d}{\simeq}\\
\Ext^*_{\add(\Proj_{\Fp}^\aleph,\kk)}(t',\rho'\circ(a\otimes -))\otimes_\kk E^*_\infty \ar{r}{\Psi_\infty} & \Ext^*_{\F(\Proj_{\Fp}^\aleph,\kk)}(t',\rho'\circ(a\otimes -))
\end{tikzcd}\;.
\end{equation*}
Thus, in order to prove the result, it suffices to prove that the lower $\Psi_\infty$ in this square is an isomorphism. By proposition \ref{pr-enlarge} again, this is equivalent to showing that 
\[\Psi_\infty: \Ext^*_{\add(\Proj_\Fp,\kk)}(t,\rho_a)\otimes_\kk E^*_\infty \to \Ext^*_{\F(\Proj_\Fp,\kk)}(t,\rho_a)\]
is an isomorphism, where $t$ and $\rho_a$ denote the restrictions of $t'$ and $\rho'\circ (a\otimes-)$ to $\Proj_\Fp$. But this is an isomorphism by lemma \ref{lm-step-1}.
\end{proof}
We can now conclude the proof of theorem \ref{thm-prelim} by a $\delta$-functor argument.
\begin{proof}[Proof of theorem \ref{thm-prelim}]
By lemma \ref{lm-step-2}, $\Psi_\infty$ is an isomorphism for $\pi=\kk\otimes_{\mathbb{Z}}\A(a,-)$ for all objects $a$ of $\A$. Now every projective object of $\add(\A,\kk)$ is a direct summand of a direct sum of such functors. Noticing that $\Ext$ sends arbitrary direct sums at the source to products, that direct sums and products are here both exact, and using degreewise finiteness of $E^*_\infty$, we see that $\Psi_\infty$ is an isomorphism for all projectives $\pi$ of $\add(\A,\kk)$.

The source and the target of $\Psi_\infty$ are $\delta$-functors of the variable $\pi$. We claim that $\Psi_\infty$ is a morphism of $\delta$-functors. 
Indeed, since $\Phi$ is a morphism of $\delta$-functors and $E_\infty^*$ is concentrated in even degrees, in order to prove our claim it suffices to prove that for all $u\in \Ext_{\F(\Proj_\kk,\kk)}^{2k}(I,I)$ the map
\[
\begin{array}{cccc}
\chi_{u}\;:\; &\Ext^*_{\F(\A,\kk)}(\pi,\rho)& \to &\Ext^{*+2k}_{\F(\A,\kk)}(\pi,\rho)\\
& x &\mapsto & x\circ (\pi^*u)
\end{array}
\] 
is a morphism of $\delta$-functors. If $u$ is represented by an extension $0\to I\to U_0\to \cdots\to U_{2k}\to I\to 0$, and if $f:\pi\to \mu$ is a morphism, the following commutative diagram (in which all the vertical arrows are induced by $f$) 
\[
\begin{tikzcd}
0\ar{r} &\pi\ar{r}\ar{d} &\pi^*U_0\ar{r}\ar{d}&\cdots \ar{r}&\pi^*U_{2k}\ar{r}\ar{d}&\pi\ar{r}\ar{d}&0\\
0\ar{r} &\mu\ar{r} &\mu^*U_0\ar{r}&\cdots \ar{r}&\mu^*U_{2k}\ar{r}&\mu\ar{r}&0
\end{tikzcd}
\]
shows that $f\circ(\pi^*u)=(\mu^*u)\circ f$ \cite[III, Proposition 5.1]{MLHom}. Therefore for all $x$ one has $x\circ (\mu^*u)\circ f= x\circ f\circ (\pi^*u)$, which shows that $\chi_{u}$ is a natural transformation. Moreover, if $e=0\to \pi\to \nu\to \mu\to 0$ is a short exact sequence of additive functors, the associated connecting morphism $\Ext^i_{\F(\A,\kk)}(\pi,\rho)\to \Ext^{i+1}_{\F(\A,\kk)}(\mu,\rho)$ sends a class $z$ to the Yoneda composite $(-1)^iz\circ e$ \cite[III Theorem 9.1]{MLHom}. Now \cite[Proposition 1.5]{FLS} shows that $(\pi^*u)\circ e=e\circ (\mu^*u)$. Hence for all $x$ of degree $i$ we have $(-1)^i x\circ (\mu^*u)\circ e= (-1)^{i+2k}x\circ e\circ (\pi^*u)$, which shows that $\chi_u$ is a morphism of $\delta$-functors.

Now we consider the following assertions:
\begin{enumerate}
\item[$P(i)$] For all $\pi$ and all $\rho$, $\Psi_\infty$ is an isomorphism in degrees less or equal to $i$.
\item[$Q(i)$] For all $\pi$ and all $\rho$, $\Psi_\infty$ is an isomorphism in degrees less or equal to $i$, and a monomorphism in degree $i+1$.
\end{enumerate}
By using left exactness of $\Hom_{\add(\A,\kk)}(\pi,\rho)$ and of $\Hom_{\F(\A,\kk)}(\pi,\rho)$ considered as functors with respect to $\pi$, together with the fact that $\Psi_\infty$ is an isomorphism for all additive projective functors $\pi$, we obtain that $P(0)$ holds. Moreover, if $q:\mu\to \pi$ is a quotient of additive functors with $\mu$ projective, by diagram chasing in the ladder induced by $\Psi_\infty$ and by the long exact sequences in $\Ext$ associated to the short exact sequence $0\to \mathrm{Ker}\,q\to \mu\to \pi\to 0$, we prove that $P(i)\Rightarrow Q(i)\Rightarrow P(i+1)$ for all $i\ge 0$. Therefore $\Psi_\infty$ is an isomorphism in all degrees.   
\end{proof}

\section{Proof of theorem \ref{thm-intro-3}}\label{sec-proof}

The following lemma gives the first part of theorem \ref{thm-intro-3}, using the computation of $E_\infty^*$ recalled in section \ref{subsec-Frobenius}.
\begin{lm}\label{lm-iso}
The map $\alpha_\infty^\kk:E_\infty^*\to \Ext^*_{\F(\Proj_\kk,\kk)}(I,I)$ is an isomorphism for all perfect fields $\kk$ of positive characteristic.
\end{lm}
\begin{proof} We are going to use a Hochschild homology computation to determine $\Ext^*_{\add(\Proj_\kk,\kk)}(I,I)$ and then apply theorem \ref{thm-prelim}.
 The enveloping ring $\kk\otimes_\mathbb{Z} \kk$ identifies as $\kk\otimes_\Fp \kk$, and every field extension of $\Fp$ is a filtered colimit of $\Fp$-subalgebras which are smooth and essentially of finite type. Thus $\mathrm{HH}_*(\kk)$ is an exterior algebra over the $\kk$-vector space of Kähler forms $\Omega^1(\kk/\Fp)$ by the Hochschild-Kostant-Rosenberg theorem \cite[Corollary 2.13]{Hubl}. Since $\kk$ is perfect, $\Omega^1(\kk/\Fp)$ is zero, so that $\mathrm{HH}_i(\kk)=\kk$ for $i=0$, and zero otherwise.
By duality, this implies that $\mathrm{HH}^i(\kk)=\kk$ for $i=0$, and zero otherwise.
Now the equivalence of categories
\[
\begin{array}{ccc}
(\kk\otimes_\Fp\kk)\Md & \simeq & \add(\Proj_\kk,\kk)\\
M & \mapsto & [v\mapsto v\otimes_\kk M]
\end{array}
\]
sends the $(\kk,\kk)$-bimodule $\kk$ to $I$, hence the graded vector space
\[\Ext^*_{\add(\Proj_\kk,\kk)}(I,I)\simeq \Ext^*_{\kk\otimes_\Fp\kk}(\kk,\kk)\simeq \mathrm{HH}^*(\kk)\]
equals zero in positive degree, and $\kk$ in degree zero. Thus, theorem \ref{thm-prelim} with $\A=\Proj_\kk$ and $\pi=\rho=I$ implies that $I^*=\alpha_\infty:E_\infty^*\to \Ext^*_{\F(\Proj_\kk,\kk)}(I,I)$ is an isomorphism. 
\end{proof}

\begin{proof}[Proof of theorem \ref{thm-intro-3}]
We observe that $\Psi_\infty = \Psi \circ (\id\otimes \alpha_\infty^\kk)$, where $\Psi_\infty$ is the isomorphism of theorem \ref{thm-prelim} and $\Psi$ is the map of theorem \ref{thm-intro-3}, which is defined in the introduction, page~\pageref{ppsi}. Since $\id\otimes \alpha_\infty^\kk$ is an isomorphism by lemma \ref{lm-iso}, we conclude that $\Psi$ is an isomorphism too.
\end{proof}

\begin{rk}\label{rk-perfect}
The proof of lemma \ref{lm-iso} shows that for $\A=\Proj_\kk$ and $\pi=\rho=I$ the source of $\Psi$ is equal to $\mathrm{HH}^*(\kk)\otimes_\kk \Ext^*_{\F(\Proj_\kk,\kk)}(I,I)$ while the target of $\Psi$ is equal to $\Ext^*_{\F(\Proj_\kk,\kk)}(I,I)$. If $\kk$ is not perfect, then $\mathrm{HH}_1(\kk)\simeq\Omega^1(\kk/\Fp)\ne 0$, so $\mathrm{HH}^*(\kk)$ is non-trivial, so that $\Psi$ fails to be an isomorphism. 
\end{rk}

\section{Dualization of theorem \ref{thm-intro-3}}
Given a vector space $V$ and a functor $F:\A^\op\to \V_\kk$, we let $D_VF:\A\to\V_\kk$ be the functor such that $D_VF(a)=\Hom_\kk(F(a),V)$. Then the tensor product over $\A$ is characterized by the isomorphism, natural with respect to $F$, $G$ and $V$:
\[\Hom_\kk(F\underset{\kk[\A]}{\otimes} G,V)\simeq \Hom_{\F(\A,\kk)}(G,D_V F)\;.\] 
By deriving this isomorphism, we obtain isomorphisms, natural with respect to $F$, $G$, $V$, in which $\Hom_\kk(-,V)$ is applied degreewise:
\[\Hom_\kk(\Tor^{\kk[\A]}_*(F,G),V)\simeq \Ext^*_{\F(\A,\kk)}(G,D_V F)\;.\]
If the functors $F$ and $G$ are additive, there is a similar isomorphism 
\[\Hom_\kk(\Tor^{\A\otimes_\mathbb{Z}\kk}_*(F,G),V)\simeq \Ext^*_{\add(\A,\kk)}(G,D_V F)\;.\]
These isomorphisms allow us to prove corollary \ref{cor-intro}.

\begin{proof}[Proof of corollary \ref{cor-intro}]
We have a chain of graded isomorphisms, natural with respect to $\pi$, $\rho$ and $V$, in which the third isomorphism is provided by theorem \ref{thm-intro-3} and the fact that $T_*$ is finite-dimensional in each degree:
\begin{align*}
\Hom_\kk(\Tor_*^{\kk\otimes_{\mathbb{Z}}\A}(\pi,\rho)\otimes_\kk T_*,V)
&\simeq \Hom_\kk(\Tor_*^{\kk\otimes_{\mathbb{Z}}\A}(\pi,\rho),V)\otimes_\kk \Ext^*_{\F(\Proj_\kk,\kk)}(I,I)\\
&\simeq \Ext^*_{\add(\A,\kk)}(\pi,D_V\rho)\otimes_\kk \Ext^*_{\F(\Proj_\kk,\kk)}(I,I)\\
&\simeq \Ext^*_{\F(\A,\kk)}(\pi,D_V\rho)\\
&\simeq \Hom_\kk(\Tor_*^{\kk[\A]}(\pi,\rho),V)\;.
\end{align*}
Hence the result follows from the Yoneda lemma.
\end{proof}

\section{Some $\Tor$-computations between non-additive functors}\label{sec-comput-tor}

Additive functors are building blocks for interesting non-additive functors, in particular polynomial functors. The simplest way to produce polynomial functors of arbitrary degree from additive functors is to consider tensor products of such functors. In this section, we explain how to compute $\Tor$ between tensor products of additive functors from the $\Tor$ between their additive building blocks. The results of the section should not be surprising for experts, though they cannot be found in the literature.
 Combined with our theorem \ref{thm-intro-3}, these results yield example \ref{ex-intro} of the introduction, and many other computations.

In this section $\A$ and $\B$ are arbitrary essentially small additive categories, and $\kk$ is an arbitrary field (of characteristic zero, or of positive characteristic). Unadorned tensor products are taken over $\kk$. We first give $\Tor$-versions of some well-known properties of $\Ext$, see e.g.  \cite[proof of Theorem 1.7]{FFSS} or \cite{Pira-Pan}. Given a functor $\phi:\A\to \B$, we have a canonical restriction map 
\[\res^\phi:\Tor_*^{\kk[\A]}(F\circ\phi, G\circ\phi)\to \Tor_*^{\kk[\B]}(F,G)\;.\]
Noting that precompositions (as $-\circ\phi$) are exact, the following lemma is proved by a straightforward verification. 
\begin{lm}\label{lm-tor-adj}
Let $\phi:\A\leftrightarrows \B:\psi$ be an adjoint pair, and let $u:\id\to \psi\circ\phi$ and $e:\phi\circ\psi\to\id$ be the unit and the counit of the adjunction. Then the following composition is an isomorphism, whose inverse is induced by $F(e)$ and $\res^\psi$:
\[\Tor_*^{\kk[\A]}(F\circ\phi,G)\xrightarrow[]{G(u)_*}\Tor_*^{\kk[\A]}(F\circ \phi,G\circ \psi\circ\phi) \xrightarrow[]{\res^\phi}\Tor_*^{\kk[\B]}(F,G\circ \psi)\;.\]
\end{lm}
In the sequel, lemma \ref{lm-tor-adj} will be applied to the diagonal functor $D:\A\to \A^d$, $D(a)=(a,\dots,a)$, and the direct sum functor $\Pi:\A^d\to \A$, $\Pi(a_1,\dots,a_d)=a_1\oplus\cdots\oplus a_d$, which are adjoint to each other on both sides, giving isomorphisms:
 \begin{align*} 
&\Tor^{\kk[\A^{d}]}_*(F\circ \Pi,G)\simeq \Tor^{\kk[\A]}_*(F,G\circ D)\;,\\
&\Tor^{\kk[\A]}_*(F'\circ D,G')\simeq \Tor^{\kk[\A^{d}]}_*(F',G'\circ \Pi)\;.
\end{align*}
The second property is a K\"unneth isomorphism. 
We let $F_1\boxtimes \cdots\boxtimes F_d:\A^d\to \V_\kk$ denote the external tensor product of functors $F_i:\A\to \V_\kk$, defined by:
\[(F_1\boxtimes \cdots \boxtimes F_d)(a_1,\dots,a_d)=F_1(a_1)\otimes\cdots\otimes F_d(a_d)\;.\]
\begin{lm}\label{lm-tor-kun}
For all functors $F_i:\A^\op\to \V_\kk$ and $G_i:\A\to \V_\kk$, with $1\le i\le d$, there is a K\"unneth isomorphism:
\[\kappa:\Tor_*^{\kk[\A^d]}(F_1\boxtimes\cdots\boxtimes F_d,G_1\boxtimes\cdots\boxtimes G_d)\simeq \bigotimes_{1\le i\le d}\Tor_*^{\kk[\A]}(F_i,G_i)\;.\]
\end{lm}
\begin{proof} For all objects $a_1,\dots,a_d$ of $\A$, the external tensor product $P_{a_1}\boxtimes \cdots \boxtimes P_{a_d}$ is isomorphic to $P_{(a_1,\dots, a_d)}$, hence the Yoneda isomorphisms yield an isomorphism, natural with respect to the functors $G_i$ and the standard projectives $P_{a_i}$:
\[(P_{a_1}\boxtimes\cdots \boxtimes P_{a_d})\underset{\kk[\A^d]}{\otimes} (G_1\boxtimes\cdots \boxtimes G_d)\simeq \bigotimes_{1\le i\le d} G_i(a_i)\simeq\bigotimes_{1\le i\le d}(P_{a_i}\underset{\kk[\A]}{\otimes} G_i)\;.\]
The result follows from this isomorphism by taking projective resolutions of the $F_i$'s by direct sums of functors of the shape $P_a$.
\end{proof}

Now we define some non-additive functors of interest.
We fix additive functors $\pi_1, \dots, \pi_r$ from $\A^\op$ to $\V_\kk$ and we consider the (non-additive if $r>1$) functor: 
\[
\begin{array}{cccc}
\pi^{\otimes}:=\pi_1\otimes\cdots \otimes\pi_r:&\A^\op&\to & \V_\kk\\
&a & \mapsto & \pi_1(a)\otimes \cdots \otimes\pi_r(a)
\end{array}\;.
\]
We also fix covariant additive functors $\rho_1,\dots,\rho_s$, we set $\rho^\otimes:=\rho_1\otimes\cdots\otimes \rho_s$, and we consider the graded vector space
\[\mathbb{T}_*=\Tor_*^{\kk[\A]}(\pi^{\otimes},\rho^{\otimes})\;.\]
The next two propositions compute this graded vector space. 
The first one is a variant of Pirashvili's vanishing lemma \cite[Corollary 2.13]{Pira-Pan}, \cite{Pira-HA}.
\begin{pr}\label{pr-vanish}
If $r\ne s$, then $\mathbb{T}_*=0$. 
\end{pr}
\begin{proof}
Every (covariant or contravariant) functor $F$ from $\A$ to $\V_\kk$ splits as a direct sum of a constant functor with value $F(0)$ and a "reduced" functor $\mathrm{red}\,F$, that is, a functor such that $(\mathrm{red}\,F)(0)=0$. Reduced functors have projective resolutions by direct sums of reduced standard projectives, and the tensor product over $\A$ of a reduced functor with a constant functor is zero, hence we obtain:

\begin{enumerate}
\item[$(*)$] Assume either that $F$ is reduced and $G$ is constant, or that $G$ is reduced and $F$ is constant. Then $\Tor_*^{\kk[\A]}(F,G)=0$.
\end{enumerate}

Now assume for example that $r> s$.
Let $\pi^{\boxtimes}:=\pi_1\boxtimes\cdots\boxtimes \pi_r$. Then $\pi^{\otimes}=\pi^{\boxtimes}\circ D$, hence lemma \ref{lm-tor-adj} yields an isomorphism $\mathbb{T}_*\simeq\Tor_*^{\kk[\A^r]}(\pi^{\boxtimes},\rho^{\otimes}\circ \Pi)$. By additivity of the $\rho_i$, the functor $\rho^{\otimes}\circ \Pi$ is isomorphic to a direct sum of external tensor products $F_1\boxtimes\cdots\boxtimes F_{r}$ and since $r> s$ at least one $F_i$ is constant. Thus, the result follows from the K\"unneth isomorphism of lemma \ref{lm-tor-kun} together with the cancellation property $(*)$.
\end{proof}

Henceforth we suppose that $r=s$. For every permutation $\sigma\in\Si_r$, we set $\pi^\otimes\sigma:=\pi_{\sigma(1)}\otimes \cdots \otimes \pi_{\sigma(r)}$ and $\pi^{\boxtimes}\sigma:=\pi_{\sigma(1)}\boxtimes\cdots\boxtimes \pi_{\sigma(r)}$. We still denote by $\sigma$ the morphism $\pi^\otimes\to\pi^\otimes\sigma$ that $\sigma$ induces:
\[
\begin{array}{cccc}
\sigma:& \pi_1(a)\otimes\cdots\otimes \pi_r(a)
&\to & \pi_{\sigma(1)}(a)\otimes\cdots\otimes \pi_{\sigma(r)}(a)\\
& x_1\otimes \cdots\otimes x_r &\mapsto & x_{\sigma(1)}\otimes \cdots\otimes x_{\sigma(r)}
\end{array}\;.
\]
We set
\[\mathbf{T}^\sigma_*:= \bigotimes_{1\le i\le r}\Tor_*^{\kk[\A]}(\pi_{\sigma(i)},\rho_i)\;,\]
and we denote by $\Upsilon$ the graded $\kk$-linear map defined as the following composition (where the first map is induced by the maps $\sigma:\pi^{\otimes}\to \pi^{\otimes}\sigma$, and the last isomorphism is the K\"unneth isomorphism of lemma \ref{lm-tor-kun}): 
\begin{align*}
\Upsilon: 
\mathbb{T}_*\xrightarrow[]{\prod \sigma_*}  
\bigoplus_{\sigma\in\Si_{r}} \Tor_*^{\kk[\A]}(\pi^{\otimes}\sigma,\rho^{\otimes})
\xrightarrow[]{\bigoplus \res^D}&
\bigoplus_{\sigma\in\Si_{r}} \Tor_*^{\kk[\A^{r}]}(\pi^{\boxtimes}\sigma,\rho^{\boxtimes})\simeq \bigoplus_{\sigma\in\Si_r}\mathbf{T}_*^\sigma\;.
\end{align*}
The next proposition is our computation of $\mathbb{T}_*$ when $r=s$.
\begin{pr}
The map $\Upsilon$ is an isomorphism.
\end{pr}
\begin{proof}
Since the $\pi_i$ are additive, there is a decomposition
\[\pi^{\otimes}(a_1\oplus\cdots\oplus a_r)=\bigoplus_{1\le i_1,\dots,i_{r}\le r} \pi_1(a_{i_1})\otimes\cdots \otimes \pi_r(a_{i_r})\;.\]
For all $\sigma\in \Si_r$ we let $f_\sigma:\pi^{\otimes}(a_1\oplus\cdots\oplus a_r)\to \pi_{\sigma(1)}(a_1)\otimes\cdots\otimes \pi_{\sigma(r)}(a_r)$ be the map which is zero on all summands of $\pi^{\otimes}(a_1\oplus\cdots\oplus a_r)$, except on the summand such that $i_k=\sigma^{-1}(k)$ for all $k$, where it is defined by $f_\sigma(x_1\otimes\cdots\otimes x_r)= x_{\sigma(1)}\otimes\cdots\otimes x_{\sigma(r)}$.
The maps $f_\sigma$ are the components of an epimorphism
\[f: \pi^{\otimes}\circ \Pi\to \bigoplus_{\sigma\in \Si_r} \pi^{\boxtimes}\sigma\;,\]
which satisfies the following properties.
\begin{enumerate}[(1)]
\item The kernel of $f$ is a direct summand of $\pi^{\otimes}\circ \Pi$, and it is isomorphic to a direct sum of external tensor products $F_1\boxtimes \cdots\boxtimes F_r$ in which at least one of the $F_i$ is constant. 
\item The map $\prod \sigma: \pi^{\otimes} \to \bigoplus_{\sigma\in\Si_r} \pi^{\otimes}\sigma$ is equal to the following composition, in which $u:\id\to \Pi\circ D$ is the unit of an adjunction between $D$ and $\Pi$:
\[\pi^{\otimes}\xrightarrow[]{\pi^{\otimes}(u)}\pi^\otimes\circ \Pi\circ D \xrightarrow[]{f(D)}\bigoplus_{\sigma\in\Si_r} \pi^{\boxtimes}\sigma\circ D=\bigoplus_{\sigma\in\Si_r} \pi^{\otimes}\sigma\;.\]
\end{enumerate}
Property (2) implies that there is a commutative diagram:
\[
\begin{tikzcd}
&\Tor^{\kk[\A]}_*(\pi^{\otimes}\circ \Pi\circ D,\rho^\otimes)\ar{r}[swap]{\res^D}\ar{d}{f(D)_*}
&\Tor^{\kk[\A^r]}_*(\pi^{\otimes}\circ \Pi,\rho^\boxtimes)\ar{d}{f_*}\\ 
\mathbb{T}_*\ar{ru}{\pi^{\otimes}(u)_*}
\ar{r}[swap]{\prod \sigma}& \Tor^{\kk[\A]}_*(\bigoplus\pi^{\otimes}\sigma,\rho^\otimes)\ar{r}[swap]{\res^D} 
&\Tor^{\kk[\A^r]}_*(\bigoplus\pi^{\boxtimes}\sigma,\rho^\boxtimes) 
\end{tikzcd}
\]
and property (1) implies (by the same reasonning as in the proof of proposition \ref{pr-vanish}) that $f_*$ is an isomorphism. Now the composition of $\res^D$ and $\pi^{\otimes}(u)_*$ is an isomorphism by lemma \ref{lm-tor-adj}, hence the bottom row of the diagram is an isomorphism. The result follows.
\end{proof}

For further computations, we assume that we are given tuples $\mathbf{r}=(r_1,\dots,r_n)$ and $\mathbf{s}=(s_1,\dots,s_m)$ such that $r_1+\cdots+r_n = r = s_1+\cdots+ s_m$, and that $\pi^{\otimes}$ and $\rho^\otimes$ have the following specific form:
\begin{align*}
&\pi^\otimes = \pi_1^{\otimes r_1}\otimes\cdots\otimes \pi_n^{\otimes r_n}\;,&\rho^\otimes = \rho_1^{\otimes s_1}\otimes\cdots\otimes \rho_m^{\otimes s_m}\;.
\end{align*}
If we denote by $\Si_{\mathbf{r}}$ the Young subgroup $\Si_{r_1}\times\cdots\times\Si_{r_n}\subset \Si_{r}$, then every $\sigma\in \Si_{\mathbf{r}}$ induces an endomorphism of $\pi^{\otimes}$, and this defines a right action of $\Si_{\mathbf{r}}$ on $\pi^{\otimes}$. Similarly, there is a right action of the Young subgroup $\Si_{\mathbf{s}}$ on $\rho^\otimes$, hence $\mathbb{T}_*$ is a right $\Si_{\mathbf{r}}\times\Si_{\mathbf{s}}$-module. The next lemma is proved by a straightforward verification.
\begin{lm}\label{lm-action} There is a right $\Si_{\mathbf{r}}\times\Si_{\mathbf{s}}$-action on $\bigoplus_{\sigma\in\Si_{r}}\mathbf{T}_*^\sigma$ with respect to which $\Upsilon$ is a morphism of $\Si_{\mathbf{r}}\times \Si_{\mathbf{s}}$-modules. It is given as follows: if $(\tau,\mu)\in\Si_{\mathbf{r}}\times\Si_{\mathbf{s}}$ and if $x=x_1\otimes \cdots \otimes x_{r}\in \mathbf{T}_*^\sigma$, then 
\[x\cdot (\tau,\mu) = \epsilon x_{\mu(1)}\otimes \cdots \otimes x_{\mu(r)}\in \mathbf{T}_*^{\tau^{-1}\sigma\mu}\;,\]
where $\epsilon$ is the usual Koszul sign determined by $\mu$ and the degrees of the $x_i$, i.e. the sign such that $x_1\cdots x_{r}=\epsilon x_{\mu(1)}\cdots x_{\mu(r)}$ in the free graded commutative algebra on $x_1,\dots,x_r$.
\end{lm}

Now let $U$ be a left $\Si_{\mathbf{r}}$-module. Then, instead of $\pi^{\otimes}$ we can consider the functor $\pi^{\otimes}\otimes_{\Si_{\mathbf{r}}} U$. Similarly if $V$ is a left $\Si_{\mathbf{s}}$-module, we consider the functor $\rho^{\otimes}\otimes_{\Si_{\mathbf{s}}} V$. The following lemma gives another $\Tor$-computation, in terms of our computation of the right $\Si_{\mathbf{r}}\times \Si_{\mathbf{s}}$-module $\mathbb{T}_*$. Note that by Maschke's theorem, the projectivity hypothesis on $U$ and $V$ is automatically satisfied if the integers $r_i$ and $s_j$ are all invertible in $\kk$ (thus if $\kk$ has characteristic zero, or big prime characteristic).
\begin{pr}\label{pr-big-char}
Assume that $U$ and $V$ are projective as modules over the symmetric group. There is a graded isomorphism, natural with respect to $U$ and $V$:
\[\Tor_*^{\kk[\A]}(\pi^{\otimes}\otimes_{\Si_{\mathbf{r}}} U,\rho^{\otimes}\otimes_{\Si_{\mathbf{s}}} V)\simeq \mathbb{T}_*\otimes_{\Si_{\mathbf{r}}\times\Si_{\mathbf{s}}} (U\otimes V)\;.\]
\end{pr}
\begin{proof}
The formula is obviously true if $U$ and $V$ are direct sums of copies of the group rings $\kk[\Si_{\mathbf{r}}]$ and $\kk[\Si_{\mathbf{s}}]$ respectively.
The general result follows from the fact that $U$ and $V$ are direct summands of such direct sums. 
\end{proof}
\begin{rk}[The relevance of $\pi^{\otimes}\otimes_{\Si_{\mathbf{r}}} U$] Let $U_i$ be a left representation  of $\Si_{r_i}$, and let $F_i:\Proj_\kk\to \V_\kk$ denote the functor such that $F_i(v)= v^{\otimes r_i}\otimes_{\Si_{r_i}}U_i$, where the right action of $\Si_{r_i}$ on $v^{\otimes r_i}$ is given by permuting the factors of the tensor product. Then if $U=\bigotimes_{1\le i\le n}U_i$ we have 
\[\pi^{\otimes}\otimes_{\Si_{\mathbf{r}}}U \simeq \pi_1^*F_1 \otimes\cdots\otimes \pi_n^*F_n\;.\]
Assume for simplicity that $\kk$ is algebraically closed. Then it follows from \cite[Theorem 4]{DTV} that if the $N_i$ are projective and simple $\Si_{e_i}$-modules, and if the additive functors $\pi_1,\dots,\pi_m$ are pairwise non-isomorphic, simple, with finite-dimensional values, then $\pi^{\otimes}\otimes_{\Si_{\mathbf{d}}}U$ is a simple functor. 
If $\kk$ has characteristic zero, such simple functors exhaust all the simple polynomial functors with finite-dimensional values (and by \cite[Theorem 3 and Corollary 4.15]{DTV} it exhausts all the simple functors with finite-dimensional values if in addition $\A$ is $\mathbb{Q}$-linear). If $\kk$ has positive characteristic, this does not exhaust all the simple polynomial functors with finite-dimensional values, but provides nonetheless many examples of simple functors. 
\end{rk}

If $\pi^{\otimes}=\pi^{\otimes r}$ for some additive functor $\pi$, then the description of the source of $\Upsilon$ can be simplified. Namely, we have an isomorphism of graded $\Si_\mathbf{r}\times\Si_{\mathbf{s}}$-modules:
\begin{align}\bigoplus_{\sigma\in\Si_r}\mathbf{T}_*^\sigma\simeq \kk[\Si_r]\otimes \mathfrak{T}_*^{\mathbf{s}}
\end{align}
where the group ring $\kk[\Si_r]$ is placed in degree zero, with action of $\Si_\mathbf{r}\times\Si_{\mathbf{s}}$ given by $\sigma\cdot (\tau,\mu)=\tau^{-1}\sigma\mu$ and $\mathfrak{T}_*^{\mathbf{s}}$ denotes the tensor product
\[\mathfrak{T}_*^{\mathbf{s}} := \Tor_*^{\kk[\A]}(\pi,\rho_1)^{\otimes s_1}\otimes\cdots\otimes \Tor_*^{\kk[\A]}(\pi,\rho_m)^{\otimes s_m}
\]
with action of $\Si_\mathbf{r}\times\Si_{\mathbf{s}}$ action given by $(x_1\otimes \cdots\otimes x_{r})\cdot (\tau,\mu) = \epsilon x_{\mu(1)}\otimes \cdots\otimes x_{\mu(r)}$ with the usual Koszul sign $\epsilon$ as in lemma \ref{lm-action}.

If $U$ equals $\kk_\mathrm{triv}$, the trivial $\Si_\mathbf{r}$-module of dimension one, then 
\[\mathbb{T}_*\otimes_{\Si_\mathbf{r}\times\Si_{\mathbf{s}}}(\kk_\mathrm{triv}\otimes V) \simeq (\mathbb{T}_*\otimes_{\Si_\mathbf{r}}\kk_\mathrm{triv})\otimes_{\Si_{\mathbf{s}}}V \simeq \mathfrak{T}_*^{\mathbf{s}}\otimes_{\Si_{\mathbf{s}}} V\;.\]
Hence proposition~\ref{pr-big-char} gives the computation of example \ref{ex-intro} (remind that $r=s$ is here assumed to be invertible in $\kk$, giving the projectivity hypothesis of proposition~\ref{pr-big-char}), as $\pi^{\otimes}\otimes_{\Si_{\mathbf{r}}}\kk_\mathrm{triv}\simeq\pi^*S^r$. Further explicit computations could be obtained with other choices of $U$.

\bibliographystyle{plain}

\bibliography{biblio-add}

\end{document}